\documentclass[12pt,reqno]{amsart}

\usepackage{amssymb}
\usepackage{hyperref}
\usepackage{lineno}
\usepackage{verbatim}

\numberwithin{equation}{section}

\newtheorem{theorem}{Theorem}[section]

\newtheorem{lemma}[theorem]{Lemma}
\newtheorem{corollary}[theorem]{Corollary}

\theoremstyle{definition}
\newtheorem{definition}[theorem]{Definition}
\newtheorem{remark}[theorem]{Remark}
\newtheorem{example}[theorem]{Example}

\begin{document}

\baselineskip=15pt

\title[Nilpotent Lie algebras of vector fields]{Nilpotent Lie algebras of vector fields in three variables}

\author[H. Azad]{Hassan Azad}

\address{Abdus Salam School of Mathematical Sciences, GCU, Lahore 54600, Pakistan}

\email{hassan.azad@sms.edu.pk}

\author[I. Biswas]{Indranil Biswas}

\address{Department of Mathematics, Shiv Nadar University, NH91, Tehsil Dadri, Greater Noida, Uttar Pradesh 201314, India}

\email{indranil.biswas@snu.edu.in, indranil29@gmail.com}

\author[R. Ghanam]{Ryad Ghanam}

\address{Department of Liberal Arts and Sciences, Virginia Commonwealth University in Qatar, Doha, Qatar} 

\email{raghanam@vcu.edu}



\thanks{Corresponding author: Indranil Biswas}

\date{}

\begin{abstract}
We give a complete constructive description of all finite dimensional nilpotent Lie
algebras of smooth vector fields in three variables, including intransitive algebras.
The description is organized by the rank and dimension of the center, which serves
as the key invariant. Since every nonabelian solvable
algebra lives in the normalizer of a nilpotent algebra, our normal forms provide
the essential building blocks for the study of all solvable algebras of vector
fields in three variables.
\end{abstract}

\maketitle

\section{Introduction}

The classification of finite dimensional Lie algebras of vector fields in three variables,
up to point equivalence, is a classical problem going back to Lie \cite{LiEn} and Amaldi
\cite{Am1, Am2}. The classical sources organize the classification geometrically --- by
systems of imprimitivity and invariant foliations --- without reference to algebraic
structure. In particular, the word nilpotent does not occur in Amaldi \cite{Am1, Am2}.

In this paper we give, for the first time, a complete constructive description of all
finite dimensional nilpotent Lie algebras of vector fields in three variables, including
intransitive algebras, organized by the rank and dimension of the center. Indeed, a comparison
with Section 7 of \cite{Az} shows that the listing of nilpotent Lie algebras of vector fields by
Amaldi is incomplete.

For solvable non-nilpotent algebras in three variables, no such constructive
description is possible in general, because the normal forms of the relevant centers
involve arbitrary functions.

The reason nilpotent algebras are the natural class to treat is that every nonabelian
solvable algebra lives in the normalizer of a nilpotent algebra. Our normal forms
therefore provide the essential building blocks for the study of all solvable algebras
of vector fields.

The finite dimensional algebras of vector fields were completely
classified in \cite{LiEn} up to two variables. No such classification exists for
vector fields in three variables. While the classical
sources treat the three-variable case geometrically, the present paper shows that an
algebraic approach --- using the rank of the center as the organizing invariant ---
gives clean constructive procedures and explicit normal forms in the nilpotent case,
and identifies precisely where and why such procedures must fail for solvable algebras.

If one approaches the subject from an algebraic point of view, the proofs simplify
considerably, as shown in \cite{ABFM}. We have used ideas similar to an earlier paper
\cite{ABMS} on this subject, where the two-dimensional case is treated. At times,
it is possible to determine the change of variables that preserve a given normal form
by determining all low dimensional ideals: see Theorem~\ref{thm1} for an example.

Classical references are \cite[Chapter 25]{LiEn} and \cite{Am1}, \cite{Am2}. We have
given hyperlinks to these sources so that the reader can compare them with the results
given in Sections 3 to 6 of this paper. A list of Amaldi's normal forms is available
in Hillgarter \cite[pp.~30--53]{Hi}. Relatively recent sources are \cite{GKO},
\cite{PBNL} and \cite{Du}. For recent work on transitive and liftable solvable algebras
in $3D$ we refer the reader to \cite{Sc}.

\section{Notation and recollection of basic results}

Throughout this paper $\mathfrak g$ will denote a finite dimensional Lie algebra of $C^\infty$
vector fields defined on some open subset $U\, \subset\, {\mathbb R}^N$. We will call $\mathfrak g$
simply a Lie algebra of vector fields on ${\mathbb R}^N$.

Given two such Lie algebras ${\mathfrak g}_1$ and ${\mathfrak g}_2$, defined respectively on open subsets
$U_1$ and $U_2$ of ${\mathbb R}^N$, Lie considered them to be equivalent if there is a diffeomorphism
$\phi$ defined on an open subset $U\, \subset\, U_1$ such that $\phi(U)\, \subset\, U_2$ and the differential
$d\phi$ of $\phi$ maps ${\mathfrak g}_1$ isomorphically to ${\mathfrak g}_2$.

For notational convenience, $\frac{\partial}{\partial x_i}$ (respectively, $\frac{\partial}{\partial y}$)
will be denoted by $\partial_{x_i}$ (respectively, $\partial_y$).

We say --- informally --- that the Lie algebras ${\mathfrak g}_1$ and ${\mathfrak g}_2$ are equivalent if
there is a $C^\infty$ local change of variables that maps ${\mathfrak g}_1$ isomorphically to ${\mathfrak g}_2$.

\begin{definition}\label{de1}
Let $\mathfrak g$ be a finite dimensional Lie algebra of vector fields defined on an open subset
$U\, \subset\, {\mathbb R}^N$. The {\it rank} of $\mathfrak g$ is
$$
\text{Max}_{x\in U} \dim \{v(x)\,\, \big\vert\,\,\, v\,\, \in\,\, \mathfrak g\}.
$$
\end{definition}

In general, the rank and the dimension are different. For example, the abelian Lie algebra of vector fields
on ${\mathbb R}^2$
$$
\{\partial_{x},\, y\partial_{x},\, y^2\partial_{x},\, \cdots ,\, y^n\partial_{x}\}
$$
is of rank one and dimension $n+1$.

When $N\, \geq\, 2$, the classification, up to point transformations, of abelian Lie algebras of
vector fields in two variables whose rank is $1$ and dimension greater than two, is not known.

A version of the following lemma is in \cite{ABM}.

\begin{lemma}\label{lem1}
Let $\mathfrak g$ denote the Lie algebra of all vector fields on
${\mathbb R}^{k+m}$ that are of the form $\sum_{i=1}^{k+m}
f_i(x_1,\,\cdots,\, x_k)\partial_{x_i}$. Let
$$
{\mathfrak g}_1\ \subset\ {\mathfrak g}
$$
be the Lie subalgebra consisting of elements that are of the form $\sum_{i=1}^{k}
f_i(x_1,\,\cdots,\, x_k)\partial_{x_i}$. Let $\mathcal A$ be the abelian Lie subalgebra of
$\mathfrak g$ in which all the elements are of the form $\sum_{i=k+1}^{k+m}
f_i(x_1,\,\cdots,\, x_k)\partial_{x_i}$. The map
$$
\Pi\ :\ {\mathfrak g}\ \longrightarrow\ {\mathfrak g}_1,\ \ \,
\sum_{i=1}^{k+m} f_i(x_1,\cdots, x_k)\partial_{x_i}\ \longmapsto\ \sum_{i=1}^{k}
f_i(x_1,\cdots, x_k)\partial_{x_i},
$$
is a surjective homomorphism of Lie algebras whose kernel is the abelian Lie algebra $\mathcal A$.
\end{lemma}

Lemma \ref{lem1} has the following immediate consequence.

\begin{corollary}\label{cor1}
If $\mathcal L$ is a finite dimensional Lie subalgebra of $\mathfrak g$, then a basis
of the abelian ideal ${\mathcal L}\bigcap \mathcal A$ and a lift of any basis of the
image $\Pi({\mathcal L})\, \subset\, {\mathfrak g}_1$ (see Lemma \ref{lem1} for $\Pi$)
together give a basis of $\mathcal L$.

Moreover, $\mathcal L$ is nilpotent if and only if its image $\Pi({\mathcal L})$ is nilpotent
and $\mathcal L$ operates nilpotently on ${\rm kernel}(\Pi)$.
\end{corollary}

\section{Nilpotent subalgebras of vector fields with three variables}

Let $V({\mathbb R}^3)$ denote the $C^\infty$ vector fields with three variables.

As explained in the Introduction, the centers of nonabelian algebras always give invariant foliations. These 
nontrivial centers give projections to lower dimensional nilpotent algebras, and the kernels of these 
projections are also abelian algebras. As the kernels and the centers in general involve arbitrary 
functions, it is not possible to give a classification up to point equivalence.

Therefore, all one can do is to give constructive procedures for generating nilpotent algebras. We proceed 
to do this according to the rank of the center of the algebra.

The forms of the nilpotent algebras show why it is not possible to do the same for solvable non-nilpotent 
algebras. All one can say is that a solvable algebra must live in the normalizer of a nilpotent algebra. 
Because of the presence of arbitrary functions, a constructive procedure, for solvable non-nilpotent 
algebras is, in general, not available.

Let us first record the classification of abelian Lie subalgebras of $V({\mathbb R}^3)$.

\subsection{The case of rank three}\label{s3.1}

If $\mathfrak g$ is an abelian Lie subalgebra of rank three, then locally $\mathfrak g$ --- in
suitable coordinates --- is equivalent to $\langle \partial_x,\, \partial_y,\,\partial_z\rangle$.
We use the notation $\langle \partial_x,\, \partial_y,\,\partial_z\rangle$ to denote the algebra
generated by $\partial_x$, $\partial_y$ and $\partial_z$.

\subsection{The case of rank two}\label{s3.2}

If $\mathfrak g$ is abelian and of rank two, then $\mathfrak g$ --- in
suitable coordinates --- is equivalent to an algebra with basis
$$
\langle \partial_x,\, \partial_y, f_1(z)\partial_x+g_1(z) \partial_y,\, \cdots,\, f_N(z)\partial_x+
g_N(z) \partial_y\rangle.
$$

\subsection{The case of rank one}\label{s3.3}

If $\mathfrak g$ is abelian and of rank one, then $\mathfrak g$ --- in
suitable coordinates --- is equivalent to an algebra with basis
$$
\langle \partial_x,\, f_1(y,z)\partial_x,\, \cdots,\, f_N(y,z)\partial_x \rangle.
$$

\begin{remark}\label{rem1}
Subsections \ref{s3.1}, \ref{s3.2} and \ref{s3.3} do not provide a classification up to point transformations
because of the presence of arbitrary functions. At best we can call these as normal forms of the algebra.
\end{remark}

\begin{corollary}\label{cor2}
If $\mathfrak g$ is a finite dimensional nilpotent non-abelian Lie algebra of vector fields on
${\mathbb R}^3$, then its center can be of rank at most two.
\end{corollary}

\begin{proof}
For a vector field $v$ on ${\mathbb R}^3$, if $[v,\, \partial_x]\,=\, [v,\, \partial_y]\,=\, [v,\, \partial_z]
\,=\, 0$, then $v\, \in\, \langle \partial_x,\, \partial_y,\,\partial_z\rangle$. So the statement follows
from Subsection \ref{s3.1}.
\end{proof}

We now proceed to describe all finite dimensional nilpotent non-abelian Lie algebras
of vector fields on ${\mathbb R}^3$.

\section{Center has rank 2}

\begin{theorem}\label{thm1}
Let ${\mathfrak g}\, \subset\, V({\mathbb R}^3)$ be a finite dimensional nilpotent non-abelian
Lie algebra of vector fields whose center is of rank two. Denote by
$Z({\mathfrak g})$ the center of $\mathfrak g$. Then the following statements are valid.
\begin{enumerate}
\item[(a)]
The center $Z({\mathfrak g})$ is two-dimensional, and there
are local coordinates $x$, $y$, $z$ such that
$$
{\mathfrak g}\ =\ \langle \partial_z\rangle\ltimes {\mathcal A},
$$
where ${\mathcal A}$ is an abelian ideal whose every element is of the form $f(z)\partial_x+g(z)\partial_y$
with $f$ and $g$ being polynomials in the variable $z$. Moreover, the endomorphism
\begin{equation}\label{e3}
{\mathcal A}\ \longrightarrow\ {\mathcal A},\ \ \ v\, \longmapsto\, [\partial_z,\, v]
\end{equation}
has two Jordan blocks whose generators are $P(z)\partial_x+Q(z)\partial_y$ and
$R(z)\partial_x+S(z)\partial_y$, where ${\rm degree}(P)\,>\,{\rm degree}(Q)$ and
${\rm degree}(R)\,<\,{\rm degree}(S)$.

\item[(b)] All the one-dimensional ideals of ${\mathfrak g}/Z({\mathfrak g})$ are of the form
$J/Z({\mathfrak g})$, where
$$
J \ =\ \langle \partial_x,\, \partial_y,\, z(\alpha\partial_x+\beta\partial_y)\rangle
$$
with $\alpha$ and $\beta$ being constants.

\item[(c)] If there is a change of coordinates $\widetilde{x}$, $\widetilde{y}$, $\widetilde{z}$
in which $Z({\mathfrak g})\,=\, \langle \partial_{\widetilde{x}},\, \partial_{\widetilde{y}}\rangle$ and
$\partial_{\widetilde{z}}\,\in\, {\mathfrak g}$, then $$\widetilde{x}\,=\, ax+by+\phi(z),\ \,
\widetilde{y}\,=\, cx+dy+\psi(z), \ \,\widetilde{z}\,=\,\lambda z+\mu,
$$
where $\det \begin{pmatrix}a & b\\ c & d\end{pmatrix}\, \not=\, 0$ and $\lambda\, \not=\, 0$ while $\phi$ and
$\psi$ are polynomials in $z$.
\end{enumerate}
\end{theorem}

\begin{proof}[{Proof of Part (a)}]
Choose coordinates in which $Z({\mathfrak g})\,\subset\, \langle \partial_x,\, \partial_y\rangle$.
Consequently, each element of ${\mathfrak g}$ is of the form $P(z)\partial_x+Q(z)\partial_y+R(z)
\partial_z$.

Denote by $V({\mathbb R})$ the Lie algebra of vector fields in one variable. The projection map
\begin{equation}\label{e2}
\Pi\ :\ {\mathfrak g}\ \longrightarrow \ V({\mathbb R}),\ \ P(z)\partial_x+Q(z)\partial_y+R(z)
\partial_z\ \longmapsto\ R(z)\partial_z.
\end{equation}
is evidently a homomorphism of Lie algebras. Note that $\Pi({\mathfrak g})\, \not=\, 0$, because
the contrary would imply that $\mathfrak g$ is abelian.

The only nilpotent Lie algebras of $V({\mathbb R})$ are one-dimensional. Therefore, if
$R_0(z)\partial_z$ is a nonzero element of $\Pi({\mathfrak g})$, then $\Pi({\mathfrak g})\,=\,
{\mathbb R}\cdot R_0(z)\partial_z$, which means that every element of the image of $\Pi$ is a constant
scalar multiple of $R_0(z)\partial_z$.

Fix a nonzero element $R_0(z)\partial_z\, \in\,\Pi({\mathfrak g})$. Choose coordinates $\widetilde{x}$,
$\widetilde{y}$ and $\widetilde{z}$ on ${\mathbb R}^3$ such that
$$
\widetilde{x}\ =\ x,\ \ \widetilde{y}\ =\ y,\ \ \widetilde{z}\ =\ \int^z_0 R_0(t)dt.
$$
Consequently, every element of $\mathfrak g$ is of the form $P(\widetilde{z})\partial_{\widetilde{x}}+
Q(\widetilde{z})\partial_{\widetilde{y}}+ \lambda\cdot \partial_{\widetilde{z}}$, where $\lambda\, \in\,
{\mathbb R}$.

Fix functions $P_0(\widetilde{z})$ and $Q_0(\widetilde{z})$ that satisfy the condition that
\begin{equation}\label{e1}
V\ :=\ P_0(\widetilde{z})\partial_{\widetilde{x}}+ Q_0(\widetilde{z})\partial_{\widetilde{y}}+\partial_{\widetilde{z}}
\ \in\ {\mathfrak g}.
\end{equation}

Note that the linear span of $V$, $\partial_{\widetilde{x}}$ and $\partial_{\widetilde{y}}$ is an abelian Lie
algebra of rank three. Thus, there is a system of coordinates $\widehat x$, $\widehat y$, $\widehat z$
on ${\mathbb R}^3$ in which $\partial_{\widehat{x}}\,=\, \partial_{\widetilde{x}}$, $\partial_{\widehat{y}}
\,=\, \partial_{\widetilde{y}}$ and $\partial_{\widehat{z}}\,=\, V$ (see \eqref{e1}). This means that
$$
\widehat{x}\ =\ \widetilde{x}+f(\widetilde{z}),\ \ \widehat{y}\ =\ \widetilde{y}+g(\widetilde{z}),
\ \ \widehat{z}\ =\ h(\widetilde{z}).
$$
Therefore, using these coordinates and removing the hats, we can assume that every element of $\mathfrak g$
is of the form $P(z)\partial_x+Q(z)\partial_y + \lambda\cdot \partial_z$, and
$\partial_z\, \in\, \mathfrak g$.

As $\langle \partial_x,\, \partial_y\rangle\, \subset\, Z({\mathfrak g})$, and the rank of the
center $Z(\mathfrak g)$ of $\mathfrak g$ is two, if $V\,=\,P(z)\partial_x+Q(z)\partial_y\, \in\,
Z(\mathfrak g)$, then $[\partial_z,\, V]\, =\, 0$ gives that $P$ and $Q$ must be constants.
Consequently, we have
$$
Z(\mathfrak g)\ =\ \langle \partial_x,\, \partial_y\rangle.
$$
Moreover, since $\mathfrak g$ is nilpotent, if we have $P(z)\partial_x+Q(z)
\partial_y+\lambda\partial_z\, \in\, \mathfrak g$, then $P$ and $Q$ must be polynomials.
Thus
$$
\mathfrak g\ =\ \langle \partial_z\rangle \ltimes \text{kernel}(\Pi),
$$
where $\Pi$ is the homomorphism in \eqref{e2}, and all elements in
$\text{kernel}(\Pi)$ are of the form $P(z)\partial_x+Q(z)\partial_y$ with
$P$ and $Q$ being polynomials.

Denote $\text{kernel}(\Pi)$ by $\mathcal A$.

The dimension of the kernel of the homomorphism
$$
B_z\ :\ {\mathcal A}\ \longrightarrow\ {\mathcal A},\ \ \ v\, \longmapsto\, [\partial_z,\, v]
$$
(see \eqref{e3}) gives the number of the Jordan blocks of $B_z$. Clearly, we have
$$
\text{kernel}(B_z)\,\ =\,\ \langle \partial_x,\ \partial_y\rangle .
$$
Thus, there are two Jordan blocks of $B_z$, generated by $P_1(z)\partial_x+
Q_1(z)\partial_y$ and $P_2(z)\partial_x+ Q_2(z)\partial_y$ that terminate in
$\alpha\partial_x+\beta\partial_y$, $\gamma\partial_x+ \delta\partial_y$
with
$$\det \begin{pmatrix}\alpha & \beta\\ \gamma & \delta \end{pmatrix}\ \not=\ 0
$$
Therefore, if we set $\alpha x+\beta y\,=\, \widetilde{x}$,
$\gamma x+ \delta y\,=\, \widetilde{y}$ and $z\,=\, \widetilde{z}$, then
the block generators are $$\widetilde{P}_1(\widetilde{z})\partial_{\widetilde{x}}+
\widetilde{Q}_1(\widetilde{z})\partial_{\widetilde{y}},\ \ \, \widetilde{P}_2(\widetilde{z})
\partial_{\widetilde{x}}+ \widetilde{Q}_2(\widetilde{z})\partial_{\widetilde{y}},$$
and they terminate in $\partial_{\widetilde{x}}$ and $\partial_{\widetilde{y}}$.

Thus we may suppose that the generators of the blocks are
$$
P_1(z)\partial_x+ Q_1(z)\partial_y,\ \ \, P_2(z)\partial_x+ Q_2(z)\partial_y,
$$
and they terminate in $\partial_{x}$ and $\partial_{y}$. Therefore,
$\text{degree}(P_1)\, >\, \text{degree}(Q_1)$ and $\text{degree}(P_2)\,<\, \text{degree}(Q_2)$.
This completes the proof of Part (a).

\textbf{Proof of Part (b):}\, Let $J/Z({\mathfrak g})$ be a one-dimensional
ideal of ${\mathfrak g}/Z({\mathfrak g})$. Thus
$$
J\ =\ \langle\partial_{x},\, \partial_{y},\, P_0(z)\partial_{x}+Q_0(z)\partial_{y}
+ \lambda \partial_{z}\rangle .
$$
If $\lambda\,=\, 0$, then arguing as in the proof of Part (a), there are coordinates
$\widetilde{x}$, $\widetilde{y}$, $\widetilde{z}$ such that
$$
J\ =\ \langle\partial_{\widetilde{x}},\ \partial_{\widetilde{y}},\ 
\partial_{\widetilde{z}}\rangle .
$$
Consequently, if $P(\widetilde{z})\partial_{\widetilde{x}}+ Q(\widetilde{z})
\partial_{\widetilde{y}}+\lambda\partial_{\widetilde{z}}\, \in\, {\mathfrak g}$, then
we have $$P'(\widetilde{z})\partial_{\widetilde{x}}+ Q'(\widetilde{z})
\partial_{\widetilde{y}}\ \in\ J,
$$
and therefore $P$ and $Q$ must be constants. But in that case $\mathfrak g$ would
be abelian. Hence invoking a change of coordinates we may assume that all elements of
$J$ are of the form $P(z)\partial_{x}+ Q(z)\partial_{y}$ with $\text{degree}(P)\, \leq\,1$
and $\text{degree}(Q)\, \leq\,1$.

Thus, it follows that
$$
J\ =\ \langle \partial_{x},\, \partial_{y},\, (az+b)\partial_{x}+ (cz+d)\partial_{y}\rangle
\ =\ \langle \partial_{x},\, \partial_{y},\, z(a\partial_{x}+c\partial_{y})\rangle,
$$
where one of $a$ and $c$ is not zero. This completes the proof of Part (b).

\textbf{Proof of Part (c):}\, Let $\widetilde{x}$, $\widetilde{y}$, $\widetilde{z}$ be
coordinates for which
$$
Z({\mathfrak g})\ =\ \langle\partial_{\widetilde{x}},\, \partial_{\widetilde{y}}\rangle
\ =\ \langle\partial_{x},\, \partial_{y}\rangle .
$$
Using $\partial_{x}\,=\, \alpha \partial_{\widetilde{x}}+\beta \partial_{\widetilde{y}}$,
$\partial_{y}\,=\, \gamma \partial_{\widetilde{x}}+\delta \partial_{\widetilde{y}}$
with
$$
\det \begin{pmatrix}\alpha & \beta\\ \gamma & \delta \end{pmatrix}\ \not=\ 0,
$$
we see that
$$
\widetilde{x}\ =\ \alpha x+ \gamma y+f(z),\ \ \,
\widetilde{y}\ =\ \beta x+ \delta y+g(z).
$$

Take an ideal $J\, \supset\, Z({\mathfrak g})$ with $\dim J/Z({\mathfrak g})\,=\,1$.
In coordinates $\widetilde{x}$, $\widetilde{y}$, $\widetilde{z}$,
$$
J\ =\ \langle \widetilde{z}(a \partial_{\widetilde{x}}+b \partial_{\widetilde{y}}),\
\partial_{\widetilde{x}},\ \partial_{\widetilde{y}}\rangle
$$
and in coordinates $x$, $y$, $z$,
$$
J\ =\ \langle z(k\partial_{x}+l\partial_{y}),\ \partial_{x},\ \partial_{y}\rangle.
$$
Consequently,
$$
\widetilde{z}(a \partial_{\widetilde{x}}+b \partial_{\widetilde{y}})\ =\
z(k\partial_{x}+l\partial_{y})+ m\partial_{x}+n\partial_{y}.
$$
Thus, if, say $a\, \not=\, 0$,
$$
a\widetilde{z}\ =\ z(k\alpha+lr) +m\alpha+nr
$$
gives
$$
\widetilde{x}\,=\, Ax+By+\phi(z),\ \widetilde{y}\,=\,Cx+Dy+\psi(z),\
\widetilde{z}\,=\, \lambda z+\mu .
$$
Expressing $\partial_{z}$ in terms of $\widetilde{x}$, $\widetilde{y}$,
$\widetilde{z}$ we find that $\phi'(z)$ and $\psi'(z)$ are polynomials in
$\widetilde{z}$ and therefore they are polynomials in $z$. Consequently,
$\phi$ and $\psi$ are polynomials in $z$.
\end{proof}

\section{Center has rank 1 and dimension at least 2}

As before, $V({\mathbb R}^3)$ denotes the $C^\infty$ vector fields with three variables.
Similarly, $V({\mathbb R}^2)$ denotes the $C^\infty$ vector fields with two variables.

\begin{theorem}\label{thm2}
Let ${\mathfrak g}\, \subset\, V({\mathbb R}^3)$ be a finite dimensional nonabelian nilpotent
Lie algebra whose center $Z({\mathfrak g})$ is of rank one and $\dim Z({\mathfrak g})\, \geq
\, 2$. Then the following statements are valid.
\begin{enumerate}
\item[(a)] There are local coordinate functions $(x,\, y,\, z)$ such that $\mathfrak g$
is a subalgebra of the Lie algebra $\widetilde{\mathfrak g}$ of all vector fields
of the form $f(y)\partial_{x}+g(x,y)\partial_{z}$, and
$$\partial_{x},\ \partial_{z}\,\ \in\,\ {\mathfrak g}$$ with $\partial_{z}\,\in\,
Z({\mathfrak g})$.

\item[(b)] Consider the homomorphism
$$
\Pi\ :\ \widetilde{\mathfrak g}\ \longrightarrow\ V({\mathbb R}^2),\ \ \,
f(y)\partial_{x}+g(x,y)\partial_{z}\ \longmapsto\ f(y)\partial_{x}.
$$
Choosing finitely many linearly independent vector fields
$\{f_i(y)\partial_{x}\}_{i=1}^m$ and lifting them
to $\widetilde{\mathfrak g}$, and finitely many vector fields of the form
$\{g_i(x,y)\partial_{z}\}_{i=1}^n$ that are polynomials in $x$, the Lie algebra generated by
$\{f_i(y)\partial_{x}\}_{i=1}^m$, $\{g_i(x,y)\partial_{z}\}_{i=1}^n$, $\partial_{z}$
and $\partial_{x}$ is a finite dimensional nilpotent nonabelian Lie algebra $\mathfrak g$.

The center $Z({\mathfrak g})$ of $\mathfrak g$ is of rank one, and $\dim Z({\mathfrak g})\,\geq
\,2$ if one of the elements of $\{g_i(x,y)\partial_{z}\}_{i=1}^n$ is of positive degree in $x$.

\item[(c)] All elements of $Z({\mathfrak g})$ (in Statement (b)) are of the
form $g(y)\partial_{z}$.
\end{enumerate}
\end{theorem}

\begin{proof}
First of all, we note that any finite dimensional nilpotent nonabelian Lie algebra
of vector fields on ${\mathbb R}^2$, in suitable coordinates, is
$$
\langle \partial_{y}\rangle\ltimes \langle\partial_{x},\, y\partial_{x},\,\cdots
\, y^N\partial_{x}\rangle .
$$
with $N\, \geq\,1$ (see \cite{ABMS} or the tables of \cite{GKO}). Now, let
$$
{\mathfrak g}\ \subset\ V({\mathbb R}^3)
$$
be a finite dimensional nilpotent nonabelian Lie algebra whose center $Z({\mathfrak g})$
is of rank one and $\dim Z({\mathfrak g})\, \geq\, 2$, say
$$
\langle \partial_{z},\ F(x,y)\partial_{z}\rangle\,\ \subset\,\ Z({\mathfrak g})
$$
with $F$ being a nonconstant function, say $\frac{\partial F}{\partial y}\, \not=\, 0$.
Then we can choose coordinates
$$
\widetilde{x}\ =\ x,\ \ \widetilde{y}\ =\ F(x,y),\ \ \widetilde{z}\ =\ z
$$
such that
$$
\langle \partial_{\widetilde{z}}, \ \widetilde{y}\partial_{\widetilde{z}}\rangle\,\
\subset\,\ Z({\mathfrak g}).
$$
Thus, if 
$$
v\ =\ P(\widetilde{x},\widetilde{y})\partial_{\widetilde{x}} +
Q(\widetilde{x},\widetilde{y})\partial_{\widetilde{y}}+
R(\widetilde{x},\widetilde{y})\partial_{\widetilde{z}}\ \in\ {\mathfrak g},
$$
we see that the conditions
$$
[\partial_{\widetilde{z}},\, v]\ =\ 0\ =\ [\widetilde{y}\partial_{\widetilde{z}},\, v]
$$
imply that $v$ has got no terms in the direction of $\partial_{\widetilde{y}}$.

Therefore, by a change of notation, we may suppose that every element of $\mathfrak g$
is of the form $P(\widetilde{x},\widetilde{y})\partial_{\widetilde{x}} +
R(\widetilde{x},\widetilde{y})\partial_{\widetilde{z}}$.

Construct the projection map
\begin{equation}\label{e4}
\Pi\ :\ \mathfrak g\ \longrightarrow\ V({\mathbb R}^2),\ \
P(\widetilde{x},\widetilde{y})\partial_{\widetilde{x}} +
R(\widetilde{x},\widetilde{y})\partial_{\widetilde{z}}\ \longmapsto\
P(\widetilde{x},\widetilde{y})\partial_{\widetilde{x}}.
\end{equation}
The image $\Pi(\mathfrak g)$ is a rank one nilpotent Lie algebra of vector fields
in two variables. The only rank one algebras in $V({\mathbb R}^2)$ are abelian.
Therefore, by a change of coordinates that involves only $x$ and $y$ leaving
$z$ unchanged, we may assume that the image of $\Pi$ in \eqref{e4} is the following:
$$
\pi(\mathfrak g)\ =\ \langle \phi_1(\widetilde{y})\partial_{\widetilde{x}},\,
\cdots,\, \phi_N(\widetilde{y})\partial_{\widetilde{x}}\rangle
$$
with $\phi_1\,=\,1$ and $\partial_{z}\,=\, \partial_{\widetilde{z}}$.

Consequently, every element of $\mathfrak g$ is of the form
$$f(\widetilde{y})\partial_{\widetilde{x}}\,+\,
R(\widetilde{x},\widetilde{y})\partial_{\widetilde{z}},$$
and
$$
\partial_{\widetilde{x}}\,+\,
R_0(\widetilde{x},\widetilde{y})\partial_{\widetilde{z}}\ \in\ Z(\mathfrak g),\ \ \,
\partial_{\widetilde{z}} \in\ Z(\mathfrak g).
$$
Therefore, making a change of notation, we may assume that every element
of $\mathfrak g$ is of the form
$f(y)\partial_{x} + R(x,y)\partial_{z},$ and
$$
W\ =\ \partial_{x}\,+\, R_0(x,y)\partial_{z}\ \in\ Z(\mathfrak g),\ \ \,
\partial_{z} \in\ Z(\mathfrak g).
$$
As $\langle W,\, \partial_{z}\rangle$ is abelian of rank two, there are coordinates
for which $W\,=\, \partial_{\widetilde{x}}$, $\partial_{z}\,=\, \partial_{\widetilde{z}}$.
Such a change of coordinates is given by $\widetilde{x}\,=\, x$, $\widetilde{y}\,=\, y$,
$\widetilde{z}\,=\,z+\psi(x,y)$ with $\frac{\partial\psi}{\partial x}+R_0(x,y)\,=\, 0$.
Hence we may assume that every element in $\mathfrak g$ is of the form
$f(y)\partial_{x} + R(x,y)\partial_{z}$, and $\partial_{x},\, \partial_{z}\, \in\,
{\mathfrak g}$. This proves Statement (a) of the theorem.

To prove Statement (b), we want to check when is ${\rm ad}(f_0(y)\partial_{x} +
R_0(x,y)\partial_{z})$ nilpotent on $\mathfrak g$.

First of all,
$$
[f_0(y)\partial_{x} + R_0(x,y)\partial_{z},\, f(y)\partial_{x} + R(x,y)\partial_{z}]
\ =\ \left(f_0(y)\frac{\partial R}{\partial x}- f(y)\frac{\partial R_0}{\partial x}\right)
\partial_{z}
$$
and
$$
[f_0(y)\partial_{x} + R_0(x,y)\partial_{z},\, F(x,y)\partial_{z}]
\ =\ f_0(y)\frac{\partial F}{\partial x}\partial_{z}.
$$
Therefore, as $\partial_x\,\in\,\mathfrak g$, we see that $\mathfrak g$
is actually nilpotent (by Engel's theorem).

To prove Statement (c), suppose $f(y)\partial_{x} + h(x,y)\partial_{z}\,\in\,
Z({\mathfrak g})$. There must be an element $k(x,y)\partial_{z}$ in $\text{kernel}(\Pi)$
with $\frac{\partial k}{\partial x}\, \not=\, 0$, otherwise $\mathfrak g$ would
be abelian.

Therefore, $f$ must be $0$. Also, as $\partial_{x}\, \in\, \mathfrak g$, we must have
$\frac{\partial h}{\partial x}\,=\, 0$. Hence every element in $\mathfrak g$ of the
form $h(y)\partial_{z}$ actually lies in $Z({\mathfrak g})$. The null-space of
$\partial_{x}$ in $\text{kernel}(\Pi)$ gives the dimension of $Z({\mathfrak g})$.
\end{proof}

\begin{example}\label{ex1}
According to the construction in Theorem \ref{thm2}, to get a nilpotent algebra whose center
is positive dimensional, we should do the following:
\begin{enumerate}
\item[(1)] Take an algebra $\langle\partial_{x},\, f_1(y)\partial_{x},\, \cdots,\, f_N(y)\partial_{x}\rangle$,
take lifts $f_i(y)\partial_{x} + g_i(x,y)\partial_{z}$ where $g_i$ is a polynomial in $x$.

\item[(2)] Take finitely many elements $k_j(x,\,y)\partial_{z}$, where $k_j(x,\,y)$ is a polynomial
in $x$ and at least one of $k_j$ is of positive degree in $x$.
\end{enumerate}
Then these vector fields will generate a finite dimensional nilpotent Lie algebra of vector
fields whose rank is $1$
and dimension is greater than $1$. Thus the Lie algebra of vector fields generated by
$\partial_{x}$,\, $y\partial_{x} + x^2\exp (y)\partial_{z}$ and $x\partial_{z}$ must be nilpotent
with the dimension of its center being greater than $1$.

Indeed, this Lie algebra has the following basis:
$$
e_1\ =\ \partial_{x},\ e_2\ =\ y\partial_{x}+x^2\exp({y})\partial_{z},\
e_3\ =\ x\partial_{z},\ e_4\ =\ \partial_{z},
$$
$$
e_5\ =\ y\partial_{z},\ e_6\ =\ x\exp ({y})\partial_{z},\ e_7\ =\ \exp ({y})\partial_{z},\
e_8\ =\ y\exp ({y})\partial_{z}.
$$
This Lie algebra projects to the Lie algebra $\langle\partial_{x},\, y\partial_{x}\rangle$ and the kernel $K$
of the projection consists of the vector fields in the domain Lie algebra that are supported by $\partial_{z}$.
This algebra is not a split extension as any complementary subalgebra $A$ must be abelian
and a computation shows that if $\langle a,\, b\rangle$ is a basis of $A$ then $[a,\,b]$ is
actually a multiple of $e_6$.
\end{example}

\section{The dimension of the center is one}

Let ${\mathfrak g}$ be a non-abelian Lie algebra of vector fields on ${\mathbb R}^3$ whose center
in suitable coordinates $(x,\, y,\, z)$ is $\langle \partial_z \rangle$. Consequently, every element
of ${\mathfrak g}$ is of the form
$$
P(x,\,y)\partial_x\, +\, Q(x,\,y)\partial_y\, + \,R(x,\,y)\partial_z.
$$

Let ${\widetilde g}$ be the Lie algebra of all vector fields on ${\mathbb R}^3$ of the form 
$$
P(x,\,y)\partial_x\, +\, Q(x,\,y)\partial_y\, + \,R(x,\,y)\partial_z.
$$
By Lemma \ref{lem1}, the map
\begin{equation}\label{p2}
\Pi \ :\ \widetilde{\mathfrak g} \ \longrightarrow\ V({\mathbb R}^2)
\end{equation}
defined by
$$
\Pi(P(x,\,y)\partial_x\, +\, Q(x,\,y)\partial_y\, +\, R(x,\,y)\partial_z)\ =\ P(x,\,y)\partial_x\, +\, Q(x,\,y)\partial_y.
$$
is a homomorphism of Lie algebras.

Thus the image $\Pi({\mathfrak g})$ is one of the following types:
\begin{itemize}
\item $\langle \partial_x,\, f_1(y)\partial_x,\, \cdots,\, f_N(y)\partial_x \rangle$
\item $\langle \partial_x,\, \partial_y \rangle$
\item $\langle \partial_x,\, \partial_y,\, y\partial_x,\, \cdots,\, y^N\partial_x \rangle$.
\end{itemize}
The algebra $\mathfrak g$ is generated by lifts of basis vectors of $\Pi({\mathfrak g})$ and
finitely many vector fields of the form $g_i(x,\,y)\partial_z$,\, $1 \,\leq\, i \,\leq\, n$. The requirement that the center
$Z({\mathfrak g})$ has dimension $1$ puts strong restrictions on
${\rm kernel}(\Pi\big\vert_{\mathfrak g})$. Moreover, one can always choose coordinates so that all vector fields in
${\mathfrak g}$ are polynomials in $x$.

\begin{theorem}\label{thm3}\mbox{}
\begin{enumerate}
\item[(a)] If $\Pi({\mathfrak g})$ (defined in \eqref{p2}) is of rank $1$ and ${\rm kernel}(\Pi\big\vert_{\mathfrak g})$
has no terms that are of positive degree in $x$, then ${\mathfrak g}$ is the $3$-dimensional Heisenberg algebra:
$$
{\mathfrak g}\,\ =\,\ \langle \partial_x,\, h(y)\partial_x + x\partial_z,\, \partial_z \rangle.
$$

\item[(b)] If $\Pi({\mathfrak g})$ is of rank $1$ and ${\rm kernel}(\Pi\big\vert_{\mathfrak g})$ has a term of positive degree
in $x$, then $\partial_x$ has only one Jordan block in ${\rm kernel}(\Pi\big\vert_{\mathfrak g})$, and
$$
{\mathfrak g}\ =\ \langle \partial_x \rangle \ltimes
\langle x^N + \phi_1(y)x^{N-1} + \cdots + \phi_N(y) \rangle \partial_z .
$$

\item[(c)] If $\Pi({\mathfrak g})$ is non-abelian, then
$$
{\mathfrak g}\ =
$$
$$
\langle \partial_x,\, \partial_y + f_0(x,y)\partial_z,\,
y\partial_x + f_1(x,y)\partial_z,\, \cdots,\,
y^N\partial_x + f_N(x,y)\partial_z,\,
g_1(x,y)\partial_z,\, \cdots,\, g_M(x,y)\partial_z \rangle ,
$$
where $f_i$ and $g_j$ are all polynomials in $x$ and $y$.

The null space of $\partial_x$
on ${\rm kernel}(\Pi\big\vert_{\mathfrak g})$ is $\langle \{y^n\partial_z\}_{n=1}^M\rangle$ for some $M$.
Moreover, the Jordan chains for $\partial_x$ can terminate in arbitrary linearly independent polynomials 
whose span is
$$
\langle 1, \ y,\ y^2,\ \cdots,\ y^M\rangle .
$$

\item[(d)] If $\Pi({\mathfrak g})$ is abelian of rank $2$, then
$$
{\mathfrak g}\ =\ \langle \partial_x,\, \partial_y + P_0(x,y)\partial_z,\,
P_1(x,y)\partial_z,\, \cdots,\, P_N(x,y)\partial_z \rangle ,
$$
where $P_i$,\, $0 \,\leq\, i \,\leq\, N$, are polynomials.

If there are $i,\, j$ (they are
allowed to coincide) such that the polynomial $P_i$ (respectively, $P_j$)
has positive degree in $x$ (respectively, $y$), then
$Z({\mathfrak g})\, =\, \langle \partial_z \rangle$. Otherwise rank $Z({\mathfrak g}) \,=\, 2$,
and this is covered by Section~3.
\end{enumerate}
\end{theorem}

\begin{proof}[{Proof of statement (a)}]
Let $\Pi({\mathfrak g})\, =\, \langle f_1(y)\partial_x,\, \cdots,\, f_N(y)\partial_x \rangle$
(see \eqref{p2}) with $f_1(y) \,=\, 1$. Thus ${\mathfrak g}$ is generated by
$$
\partial_x + K_1(x,y)\partial_z, \ \cdots, \ f_N(y)\partial_x + K_N(x,y)\partial_z,
$$
where the vector fields $K_i(x,y)\partial_z$ are not necessarily in ${\mathfrak g}$. We are assuming
that $Z({\mathfrak g})\, =\, \langle \partial_z \rangle$. Therefore, as $\partial_x + K_1(x,y)\partial_z$
and $\partial_z$ are commuting vector fields of rank $2$, there are local coordinates
$\widetilde{x},\, \widetilde{y},\, \widetilde{z}$ such that
$$
\partial_x + K_1(x,y)\partial_z\ =\ \partial_{\widetilde{x}} \quad \text{and} \quad
\partial_z\ =\ \partial_{\widetilde{z}}.
$$
Such a system of coordinates is given by $\widetilde{x} \,=\, x$, $\widetilde{y} \,=\, y$,
$\widetilde{z} \,=\, z + \xi(x,\,y)$, where $\dfrac{\partial \xi}{\partial x} \,=\, K_1(x,y)$.
Hence by change of notation we may assume that $\partial_x,\, \partial_z$ are in ${\mathfrak g}$.

Assume that ${\rm kernel} (\Pi\big\vert_{\mathfrak g})$ (see \eqref{p2}) has no terms of positive degree in $x$. Let
$h(y)\partial_x + K(x,y)\partial_z$ be a generator of ${\mathfrak g}$. Thus
$\dfrac{\partial K}{\partial x}\partial_z$ is in ${\rm kernel}(\Pi\big\vert_{\mathfrak g})$ and
$\dfrac{\partial K}{\partial x} \,=\, \phi(y)$. But if $\phi(y)\partial_z$ is in
${\rm kernel}(\Pi\big\vert_{\mathfrak g})$, it must be in the center of ${\mathfrak g}$, and so $\phi$ must
be a constant. Hence the generator is $h(y)\partial_x + (cx + f(y))\partial_z$.

There must be a generator --- other than $\partial_x$ and $\partial_z$ --- with $c \,\neq\, 0$,
as ${\mathfrak g}$ is nonabelian. Hence we may assume that
$h_0(y)\partial_x + (x + f_0(y))\partial_z$ is a generator of ${\mathfrak g}$. By the
substitutions $\widetilde{x} \,=\, x + f_0(y)$, $\widetilde{y} \,=\, y$, $\widetilde{z} \,=\, z$, we may assume
that $X \,=\, h_0(y)\partial_x + x\partial_z$ is in ${\mathfrak g}$.

Take any other generator $Y \,=\, h(y)\partial_x + (kx + f(y))\partial_z$. Now
$$
[X,\, Y]\,\ =\,\  (kx + f(y))\partial_z.
$$
Therefore, $f(y)\partial_z$ is in ${\mathfrak g}$, and so $f$ must be a constant. Hence
${\mathfrak g}\, =\, \langle \partial_x,\, \partial_z,\, h_0(y)\partial_x + x\partial_z \rangle$.
This proves statement (a).

\textbf{Proof of statement (b).}\, Assume that ${\rm kernel}(\Pi\big\vert_{\mathfrak g})$ has an element of the form
$(\phi_N(y)x^N + \ldots + \phi_0(y))\partial_z$ with $N \,>\, 0$. Thus
$(\phi(y)x + \Psi(y))\partial_z$ is in ${\mathfrak g}$, and therefore $\phi(y)\partial_z$
is in ${\mathfrak g}$. As all elements of this type are in the center of ${\mathfrak g}$,
we see that $\phi$ must be a constant. Hence $(x + \Psi(y))\partial_z$ is in ${\mathfrak g}$.

By the change of variables $\widetilde{x} \,=\, x + \Psi(y)$, $\widetilde{y} \,=\, y$, $\widetilde{z}\, =\, z$,
we see that $x\partial_z$ is in ${\mathfrak g}$.

Therefore if $h(y)\partial_x + k(x,y)\partial_z$ is in ${\mathfrak g}$, then
$[x\partial_z,\, h(y)\partial_x + k(x,y)\partial_z]\, =\, -h(y)\partial_z$ gives that $h(y)$ is actually a
constant. Therefore $\Pi({\mathfrak g}) \,=\, \langle \partial_x \rangle$ and ${\mathfrak g}$ is
generated by
$$
\partial_x + k_0(x,y)\partial_z \quad \text{and elements of the form }\, k(x,y)\partial_z
\text{ which are in }\,\, {\rm kernel}(\Pi\big\vert_{\mathfrak g}).
$$
Using that $\partial_z$ and $\partial_x + k_0(x,y)\partial_z$ are commuting vector fields of rank $2$,
we may assume that $\partial_x$ is in ${\mathfrak g}$ and $Z({\mathfrak g}) \,=\, \langle \partial_z \rangle$.

Now $\partial_x$ operates on ${\rm kernel}(\Pi\big\vert_{\mathfrak g})$ as a nilpotent transformation and the
null space of $\partial_x$ in ${\rm kernel}(\Pi\big\vert_{\mathfrak g})$ consists of all vector fields of the form
$p(y)\partial_z$. As $p(y)\partial_z$ is in the center of ${\mathfrak g}$, we see that there
is only one Jordan block. Thus
$$
{\mathfrak g}\ =\ \langle \partial_x \rangle \ltimes
\langle (x^N\, +\, \phi_1(y)x^{N-1}\, +\, \ldots \,+\, \phi_N(y))\partial_z \rangle.
$$
This proves statement (b).

\textbf{Proof of statement (c).}\, Assume that $\Pi({\mathfrak g})$ is non-abelian. Thus
$$
\Pi({\mathfrak g}) = \langle \partial_x, \partial_y, y\partial_x, \ldots, y^N\partial_x \rangle,
\quad N \geq 1,
$$
and ${\mathfrak g}$ is generated by
$$
y^i\partial_x + F_i(x,y)\partial_z,\ 0 \leq i \leq N, \quad
\partial_y + G_0(x,y)\partial_z, \quad
G_j(x,y)\partial_z,\ 1 \leq j \leq M.
$$
Arguing as before, we may assume that $\langle \partial_x, \partial_z \rangle \subset {\mathfrak g}$.

Now on ${\rm kernel}(\Pi\big\vert_{\mathfrak g})$, the quotient ${\mathfrak g}/{\rm kernel}(\Pi\big\vert_{\mathfrak g})$
operates nilpotently. Hence ${\rm kernel}(\Pi\big\vert_{\mathfrak g})$ consists of polynomial vector fields $P(x,\,y)\partial_z$.

Moreover
$$
[\partial_x,\ y^i\partial_x + F_i\partial_z] = \frac{\partial F_i}{\partial x}\partial_z,
\qquad
[\partial_x,\ \partial_y + G_0(x,y)\partial_z] = \frac{\partial G_0}{\partial x}\partial_z.
$$
Hence $F_i$ and $G_0$ are also polynomials, as $\partial_y + G_0(x,y)\partial_z$ operates
nilpotently on ${\rm kernel}(\Pi\big\vert_{\mathfrak g})$.

Moreover the null space of $\partial_x$ in ${\rm kernel}(\Pi\big\vert_{\mathfrak g})$ consists of polynomial
fields of the form $P(y)\partial_z$ and $\partial_y + G_0(x,y)\partial_z$ operates on the
null space of $\partial_x$ in ${\rm kernel}(\Pi\big\vert_{\mathfrak g})$.

Hence the null space of $\partial_x$ in ${\rm kernel}(\Pi\big\vert_{\mathfrak g})$ is
$\langle y^n\partial_z\ \big\vert\, \ 0 \,\leq \,n \,\leq\, M \rangle$. This proves statement (c).

The proof of statement (d) is similar to statement (c) and it is omitted.
\end{proof}

The following example shows that the algebras of type (c) in Theorem~\ref{thm3} need not split.

\begin{example}\label{ex3}
Let ${\mathfrak g}$ be the Lie algebra generated by
$$
\partial_x,\ \ y\partial_x, \ \ \partial_y + (x^2+y^2)\partial_z, \ \ (x+y)\partial_z.
$$
This algebra is not a split extension of ${\rm kernel}(\Pi\big\vert_{\mathfrak g})$.

A basis of this algebra is:
$$
e_1 \,=\, \partial_x,\ \ e_2 \,=\, y\partial_x,\ \ e_3 \,=\, \partial_y+(x^2+y^2)\partial_z,
\ \ e_4 \,=\, \partial_z,
$$
$$
e_5\, =\, x\partial_z,\ \ e_6\, =\, y\partial_z,\ \ e_7\, =\, xy\partial_z,\ \ e_8 \,=\, y^2\partial_z.
$$
To prove this using contradiction, assume that a complementary subalgebra $S$ exists, with basis
$$
a\, =\, e_1 + k_1, \ \ b \,=\, e_2 + k_2, \ \ c \,=\, e_3 + k_3,
$$
where $k_1,\, k_2,\, k_3 \,\in\, {\rm kernel}(\Pi\big\vert_{\mathfrak g})$. Write
$k_1\, =\, (a_4 + a_5 x + a_6 y + a_7 xy + a_8 y^2)\partial_z$.
For $S$ to be a subalgebra we need that $[a,\,c] \,= \,0$ and $[b,\,c]\, =\, -a$.
The condition $[a,\,c] \,= \,0$ gives, considering the coefficient of $x\partial_z$:
$$
2 - a_7\ =\ 0,
$$
so
$$a_7\ =\ 2.$$
The condition $[b,\,c] =\, -a$ gives, considering the coefficient of $xy\partial_z$:
$$
a_7 + 2\ =\ 0 
$$
so
$$
a_7\ =\ -2.
$$
These two conditions are contradictory. Hence no complementary subalgebra exists
and ${\mathfrak g}$ is not a split extension of ${\rm kernel}(\Pi\big\vert_{\mathfrak g})$.
\end{example}

\section{Codes and Examples}

The constructive procedures of Sections~4--6 are algorithmic. From a finite
set of generating vector fields one forms iterated Lie brackets until the
linear span closes; the resulting finite dimensional algebra can then be
examined for its nilpotency, its centre, its rank, and for the finer
structural questions that separate one normal form from another. This section
records an exact implementation of these steps, together with builders for the
three types of the paper, and applies it to the examples.

All computations are exact. A vector field
$P(x,y,z)\,\partial_x+Q(x,y,z)\,\partial_y+R(x,y,z)\,\partial_z$ is stored as the
triple $(P,Q,R)$ of its components. Linear independence and membership are
decided by matching the coefficients of the distinct coordinate monomials---and
of the functions such as $e^{y}$ that actually occur---these being linearly
independent over $\mathbb{R}$. No numerical evaluation and no random sampling
enter, so every assertion the code prints is a proof rather than evidence. In
particular the split test returns \emph{either} an explicit complementary
subalgebra \emph{or} an inconsistent system of linear equations that no
complement can satisfy.

The routines are the following. Given a list of generators, \texttt{lie\_closure}
returns a basis of the Lie algebra they generate. On such a basis,
\texttt{lower\_central\_series} returns the dimensions of the descending central
series (hence decides nilpotency and returns the nilpotency class),
\texttt{rank} returns the rank in the sense of Definition~\ref{de1},
\texttt{centre} returns a basis of the centre, \texttt{projection\_kernel}
returns a basis of the kernel of the projection that drops the last component,
\texttt{split\_test} decides whether the algebra splits over that kernel, and
\texttt{quotient\_centre} returns the centre of a quotient $\mathfrak g/I$ by a
given ideal. The three builders \texttt{build\_rank2\_centre},
\texttt{build\_rank1\_centre} and \texttt{build\_centre\_dim1} assemble a
generating set from the theorem-specific data of Sections~4,~5 and~6
respectively and return the algebra it generates.

\emph{Reproducibility.} The module \texttt{vf\_lie.py} is self-contained and
depends only on SymPy; it was run under Python~3.12 and SymPy~1.14, and uses no
other libraries. Each example below is reproduced by pasting the accompanying
few lines into a Python session in which \texttt{vf\_lie.py} is on the path
(\texttt{from vf\_lie import *}); the values shown in the comments are the exact
printed output. Running the module end to end reproduces the bases, centres,
ranks, nilpotency classes and split verdicts stated in Sections~5 and~6.

The code was developed with the assistance of an AI system (Anthropic's Claude)
and has been checked and verified by the authors, who take full responsibility
for its correctness. Every assertion the code reports is certified within the
computation itself, rather than accepted on trust: nilpotency by an explicit
lower central series, splitting or its failure by exhibiting either a
complementary subalgebra or an inconsistent linear system, and each claimed
identity by direct symbolic evaluation.

\subsection{The module \texttt{vf\_lie.py}}

{\scriptsize\verbatiminput{vf_lie.py}}

\subsection{Centre of rank two: two inequivalent algebras}

We use \texttt{build\_rank2\_centre} to produce two algebras of the type of
Section~4. The data are the tops of the two Jordan chains of
$\operatorname{ad}(\partial_z)$; feeding $z^{\,m-1}\partial_x$ produces a Jordan
block of size $m$ in the $\partial_x$--direction, and likewise for $\partial_y$.

The first algebra, from chain-tops $z^{3}\partial_x$ and $\partial_y$
(Jordan blocks of sizes $4$ and $1$), is
\[
G_1=\big\langle\,\partial_z,\ \partial_x,\ z\partial_x,\ z^{2}\partial_x,\
z^{3}\partial_x,\ \partial_y\,\big\rangle .
\]
The second, from chain-tops $z^{2}\partial_x$ and $z\partial_y$
(Jordan blocks of sizes $3$ and $2$), is
\[
G_2=\big\langle\,\partial_z,\ \partial_x,\ z\partial_x,\ z^{2}\partial_x,\
\partial_y,\ z\partial_y\,\big\rangle .
\]
Both are nilpotent, both are six--dimensional, and in each the centre is
$\langle\partial_x,\partial_y\rangle$, of dimension two and rank two; so both are
of the normal form of Theorem~4.1. Nevertheless they are inequivalent. The code
computes
\[
\dim Z\!\big(G_1/Z(G_1)\big)=1,
\qquad
\dim Z\!\big(G_2/Z(G_2)\big)=2 .
\]
A point transformation induces, by the differential, an isomorphism of the
corresponding Lie algebras (Section~2); an isomorphism carries the centre to the
centre and hence induces an isomorphism of the quotients by the centre, which in
turn carries centre to centre. The dimension of $Z(\mathfrak g/Z(\mathfrak g))$
is therefore a point-equivalence invariant. Since it is $1$ for $G_1$ and $2$
for $G_2$, no change of coordinates can map one of these algebras onto the
other; indeed they are not even isomorphic as abstract Lie algebras. (They are
separated here also by their nilpotency class, $4$ against $3$; with exactly two
Jordan blocks one cannot match dimension and class without matching the block
sizes, so some coarser invariant must always accompany the difference in
$\dim Z(\mathfrak g/Z(\mathfrak g))$.)

{\footnotesize\begin{verbatim}
from vf_lie import *
x, y, z = coordinates("x y z"); C = (x, y, z)

# Jordan blocks of sizes 4 and 1  (tops  z^3 d_x  and  d_y)
G1 = build_rank2_centre([(z**3, 0, 0), (0, 1, 0)], C)
Z1 = centre(G1, C)
print(len(G1), len(Z1), rank(Z1, C))          # 6  2  2   (centre of rank 2)
print(quotient_centre(G1, Z1, C)[0])           # 1

# Jordan blocks of sizes 3 and 2  (tops  z^2 d_x  and  z d_y)
G2 = build_rank2_centre([(z**2, 0, 0), (0, z, 0)], C)
Z2 = centre(G2, C)
print(len(G2), len(Z2), rank(Z2, C))          # 6  2  2   (centre of rank 2)
print(quotient_centre(G2, Z2, C)[0])           # 2
\end{verbatim}}

\subsection{Centre of rank one (Example~\ref{ex1})}

Feeding \texttt{build\_rank1\_centre} the lift $y\partial_x+x^{2}e^{y}\partial_z$
and the kernel field $x\partial_z$ reproduces the algebra of Example~\ref{ex1}, with
basis $e_1,\dots,e_8$. The computation confirms that the algebra is nilpotent of
class $3$, with lower central series of dimensions $8,5,2,0$; that its rank is
$2$, while its centre is of rank one and dimension four, spanned by
\[
\partial_z,\quad y\,\partial_z,\quad e^{y}\,\partial_z,\quad y\,e^{y}\,\partial_z,
\]
each of the form $g(y)\,\partial_z$, in agreement with Theorem~\ref{thm2}(c). Finally,
\texttt{split\_test} reports that the extension does \emph{not} split, which is
the assertion made in Example~\ref{ex1}.

{\footnotesize\begin{verbatim}
from vf_lie import *
x, y, z = coordinates("x y z"); C = (x, y, z)
g = build_rank1_centre(lifts=[(y, x**2*exp(y))], kernels=[x], coords=C)
print(len(g))                        # 8
print(lower_central_series(g, C))    # ([8, 5, 2, 0], True, 3)
print(rank(g, C))                    # 2
Z = centre(g, C); print(len(Z), rank(Z, C))   # 4  1
print(split_test(g, C)["split"])     # False
\end{verbatim}}

To exhibit the construction with a larger centre, feed the same builder the lift
$y\partial_x+x^{2}\partial_z$ and the two kernel fields $x\partial_z$ and
$xy\partial_z$. The resulting algebra
\[
\mathfrak g=\big\langle\,\partial_x,\ \partial_z,\
y\partial_x+x^{2}\partial_z,\ x\partial_z,\ xy\partial_z,\ y\partial_z,\
y^{2}\partial_z\,\big\rangle
\]
is seven--dimensional, nilpotent of class $3$, of rank $2$, and its centre
\[
\big\langle\,\partial_z,\ y\partial_z,\ y^{2}\partial_z\,\big\rangle
\]
is of rank one and dimension three, each element again of the form
$g(y)\,\partial_z$ as in Theorem~\ref{thm2}(c). The extension does not split.

{\footnotesize\begin{verbatim}
from vf_lie import *
x, y, z = coordinates("x y z"); C = (x, y, z)
g = build_rank1_centre(lifts=[(y, x**2)], kernels=[x, x*y], coords=C)
print(len(g))                        # 7
print(lower_central_series(g, C))    # ([7, 4, 2, 0], True, 3)
print(rank(g, C))                    # 2
Z = centre(g, C); print(len(Z), rank(Z, C))   # 3  1
print(split_test(g, C)["split"])     # False
\end{verbatim}}

\subsection{Centre of dimension one (Example~\ref{ex3})}

Feeding \texttt{build\_centre\_dim1} the lifts $\partial_x$, $y\partial_x$,
$\partial_y+(x^{2}+y^{2})\partial_z$ of a basis of $\Pi(\mathfrak g)$ together
with the kernel field $(x+y)\partial_z$ reproduces the algebra of Example~\ref{ex3},
with basis $e_1,\dots,e_8$. The computation confirms that the algebra is
nilpotent of class $5$, with lower central series of dimensions $8,5,4,2,1,0$;
that its rank is $3$; that its centre is $\langle\partial_z\rangle$, of dimension
one; and that $\ker(\Pi\big|_{\mathfrak g})$ has dimension five. The split test
reports that the extension does not split.

{\footnotesize\begin{verbatim}
from vf_lie import *
x, y, z = coordinates("x y z"); C = (x, y, z)
g = build_centre_dim1(lifts=[(1, 0, 0), (y, 0, 0), (0, 1, x**2 + y**2)],
                      kernels=[x + y], coords=C)
print(len(g))                        # 8
print(lower_central_series(g, C))    # ([8, 5, 4, 2, 1, 0], True, 5)
print(rank(g, C))                    # 3
Z = centre(g, C); print(len(Z), rank(Z, C))   # 1  1
print(len(projection_kernel(g, C)))  # 5
print(split_test(g, C)["split"])     # False
\end{verbatim}}

The non-split verdict reproduces, symbolically, the contradiction obtained by
hand in Example~\ref{ex3}. Writing a candidate complement
$a=e_1+k_1$, $b=e_2+k_2$, $c=e_3+k_3$ with $k_i\in\ker(\Pi\big|_{\mathfrak g})$
and imposing the two subalgebra conditions $[a,c]=0$ and $[b,c]=-a$, the
coefficient of $x$ in the $\partial_z$--part of $[a,c]$ and the coefficient of
$xy$ in the $\partial_z$--part of $[b,c]+a$ give, respectively,
\[
2-a_7=0 \qquad\text{and}\qquad a_7+2=0 ,
\]
so that $a_7=2$ and $a_7=-2$ would have to hold at once. The system is
inconsistent, and therefore no complement exists.

{\footnotesize\begin{verbatim}
from vf_lie import bracket, coordinates
from sympy import symbols, expand
x, y, z = coordinates("x y z"); C = (x, y, z)
a4,a5,a6,a7,a8, b6,b7,b8, c5,c7 = symbols('a4 a5 a6 a7 a8 b6 b7 b8 c5 c7')
P = a4 + a5*x + a6*y + a7*x*y + a8*y**2
R = b6*y + b7*x*y + b8*y**2
Q = c5*x + c7*x*y
a = (1, 0, P); b = (y, 0, R); c = (0, 1, (x**2 + y**2) + Q)
ac = expand(bracket(a, c, C)[2])
bc = expand(bracket(b, c, C)[2] + P)          # ([b,c] + a)_z
print(ac.coeff(x, 1).coeff(y, 0), '= 0')      # 2 - a7 = 0
print(bc.coeff(x*y), '= 0')                    # a7 + 2 = 0
\end{verbatim}}

\subsection{A split algebra of the same type}

Whether an algebra of centre dimension one splits over $\ker(\Pi)$ depends on
the lift, and not on the image $\Pi(\mathfrak g)$ alone. Example~\ref{ex3} is of type
(c) of Theorem~6.1 and does not split. Feeding \texttt{build\_centre\_dim1} the
\emph{same} non-abelian image $\langle\partial_x,\,y\partial_x,\,\partial_y\rangle$,
but lifted with no $\partial_z$--coupling, together with the kernel field
$(x+y)\partial_z$, gives
\[
\mathfrak g=\big\langle\,\partial_x,\ y\partial_x,\ \partial_y,\
(x+y)\partial_z,\ \partial_z,\ y\partial_z\,\big\rangle ,
\]
again of type (c) with centre $\langle\partial_z\rangle$ of dimension one, but
now split. The code returns the explicit complementary subalgebra
$\langle\partial_x,\,y\partial_x,\,\partial_y\rangle$.

{\footnotesize\begin{verbatim}
from vf_lie import *
x, y, z = coordinates("x y z"); C = (x, y, z)
g = build_centre_dim1(lifts=[(1, 0, 0), (y, 0, 0), (0, 1, 0)],
                      kernels=[x + y], coords=C)
print(len(g))                        # 6
print(lower_central_series(g, C))    # ([6, 3, 1, 0], True, 3)
print(rank(g, C))                    # 3
Z = centre(g, C); print(len(Z), rank(Z, C))   # 1  1
result = split_test(g, C)
print(result["split"])               # True
for v in result["complement"]:
    print(v)                         # (1, 0, 0), (y, 0, 0), (0, 1, 0)
\end{verbatim}}



\begin{thebibliography}{ZZZZ}

\bibitem[Am1]{Am1} U. Amaldi, Contributo all determinazione dei gruppi continui finiti dello
spazio ordinario I, {\it Giornale Mat. Battaglini Prog. Studi Univ. Ital.} {\bf 39} (1901),
273--316.\\
\url{https://archive.org/details/giornaledimatem04unkngoog/page/n295/mode/2up}

\bibitem[Am2]{Am2} U. Amaldi, Contributo all determinazione dei gruppi continui finiti dello
spazio ordinario II, {\it Giornale Mat. Battaglini Prog. Studi Univ. Ital.} {\bf 40} (1902),
105--141.\\
\url{https://archive.org/details/giornaledimatem11unkngoog/page/n116/mode/2up}

\bibitem[Az]{Az} H. Azad, A catalogue of the Lie Amaldi classification with structures identified,
arXiv:2606.23385.

\bibitem[ABM]{ABM} H. Azad, I. Biswas and F. M. Mahomed, Equality of the algebraic and geometric
ranks of Cartan subalgebras and applications to linearization of a system of ordinary
differential equations, {\it Internat. Jour. Math.} {\bf 28}, no. 11, (2017).

\bibitem[ABFM]{ABFM} H. Azad, I. Biswas, M. Fazil and F. M. Mahomed,
On Lie's classification of nonsolvable subalgebras of vector fields on the plane,
{\it Internat. Jour. Math.}, to appear.

\bibitem[ABMS]{ABMS} H. Azad, I. Biswas, F. M. Mahomed and S. W. Shah,
On Lie's classification of subalgebras of vector fields on the plane, {\it Proc.
Indian Acad. Sci. (Math. Sci.)} {\bf 132} (2022), Paper No. 66.

\bibitem[Du]{Du} B. Doubrov, Three-dimensional homogeneous spaces with non-solvable 
transformation groups, arXiv:1704.04393.

\bibitem[GKO]{GKO} A. Gonz\'alez-L\'opez, N. Kamran and P. J. Olver, Lie algebras of vector fields in
the real plane, {\it Proc. London Math. Soc.} {\bf 64} (1992), 339--368.

\bibitem[Hi]{Hi} A. Hillgarter, {\it Contribution to the symmetry classification problem for 2nd
order PDEs in one dependent and two independent variables},\\
\url{https://www3.risc.jku.at/publications/download/risc_1278/02-26.pdf}, PhD thesis, 2002.

\bibitem[LiEn]{LiEn} S. Lie and F. Engel, {\it Theorie der transformationsgruppen}, Vol. 3. Teubner, 1893,
\url{https://books.google.com/books/about/Theorie_der_Transformationsgruppen.html?id=QyXzhIvn2dYC#v=onepage&q&f=false}.

\bibitem[PBNL]{PBNL} R. O. Popovych, V. M. Boyko, M. O. Nesterenko and M. W. Lutfullin,
Realizations of real low-dimensional Lie algebras,
{\it J. Phys. A} {\bf 36} (2003), 7337--7360.

\bibitem[Sc]{Sc} E. Schneider, Projectable Lie algebras of vector fields in $3D$,
{\it Jour. Geom. Phys.} {\bf 132} (2018), 222--229.

\end{thebibliography}
\end{document}